\documentclass{tac}
\usepackage{subfiles}
\usepackage{amssymb, amsmath}
\usepackage[backend=bibtex,citestyle=authoryear-icomp]{biblatex}
\usepackage[all]{xy}
\usepackage{url}

\title{Constructive Simplicial Homotopy}
\author{Wouter Pieter Stekelenburg}
\copyrightyear{2015,2016}
\address{Faculty of Mathematics, Informatics and Mechanics\\
University of Warsaw\\
Banacha 2\\
02-097 Warszawa\\
Poland}
\eaddress{w.p.stekelenburg@gmail.com}
\keywords{realizability, simplicial homotopy, Kan complexes}
\amsclass{03D80, 18G30, 18G55}

\newcommand\hide[1]{}
\newcommand\cat\mathcal
\newcommand\set[1]{\left\{#1\right\}}
\mathrmdef{id}
\mathrmdef{dom}
\mathrmdef{cod}
\newcommand\ri{^*}
\newcommand\N{\mathbb N}
\mathbfdef[nno]{N}
\newcommand\dual{^{\mathrm{op}}}
\mathbfdef[simCat]\Delta
\newcommand\s{^{\simCat\dual}}
\mathssbxdef{Asm}
\newcommand\bang{!}
\newcommand\of{:}
\newcommand\simplex\Delta
\newcommand\cycle{\partial\Delta}
\newcommand\horn\Lambda
\newcommand\f{_f}
\newcommand\tuplet[1]{\left\langle #1 \right\rangle}
\newcommand\true{\mathtt{true}}
\newcommand\false{\mathtt{false}}
\newcommand\bool{\mathtt{bool}}
\mathrmdef{nat}
\mathssbxdef{Ar}
\mathssbxdef{Ob}
\newcommand\pp{\mathbin\diamond}
\newcommand\norm[1]{\Vert #1 \Vert}
\newcommand\ka\kappa
\newcommand\la\lambda
\mathrmdef{face}
\mathrmdef{colim}
\newcommand\ex{_{\textrm{ex}}}
\newcommand\citep[1]{[\cite{#1}]}
\mathrmdef{dim}
\newcommand\base{\mathbf{U}}
\mathssbxdef[sub]{Sub}
\newcommand\ambient{\mathfrak A}
\mathrmdef{uni}
\newcommand\disc{_{\rm disc}}
\mathrmdef{filler}
\newcommand\traco\omega

\newcommand\product[2]{\Pi #1 \mapsto #2}
\newcommand\coproduct[2]{\Sigma #1 \mapsto #2}
\newcommand\function[2]{\lambda #1 \mapsto #2}

\newcommand\keyword[1]{\emph{#1}\label{#1}}

\begin{document}

\begin{abstract} This paper aims to help the development of new models of homotopy type theory, in particular with models that are based on realizability toposes. For this purpose it develops the foundations of an internal simplicial homotopy that does not rely on classical principles that are not valid in realizability toposes and related categories.\end{abstract}

\hide{
Three papers:
-simplicial homotopy
-complete categories [how they are preserved]
-the realizability model of HOTT [how to get a fibrant object out of a category]

Idee: reverse the order. definitions--theorem--lemmas. That way the purpose of the lemmas is set up from the start.
}

\maketitle

\section*{Introduction}
This paper grew out of an attempt to build a \emph{recursive realizability} model for \emph{homotopy type theory} following the example of \citep{KLV12}. Their model is based on the category of simplicial sets. Simplicial homotopy is the homotopy of simplicial sets. It is equivalent to the homotopy of CW complexes which are a category of topological spaces \citep{Hovey99,GJSHT}.

In the intended model types are homotopy types of certain simplicial objects in a \emph{realizability topos} (see \citep{MR2479466}). These toposes lack some of the structure of the category of sets underlying classical homotopy theory. In the internal logic the principle of the excluded middle and the axiom of choice can be invalid. They can lack infinite limits and colimits. We work around these issues in the following way.
\begin{enumerate}
\item Limit the class of cofibrations. In classical simplicial homotopy theory every mono\-morphism is a cofibration. We demand that certain properties of the monomorphisms are decidable so classical arguments remain valid.
\item Strengthen the lifting property. Fibrations come equipped with a \emph{filler operator} that gives solutions for a family of basic lifting problems.
\item Build the homotopy category out of fibrant objects only, so we don't need fibrant replacements.
\end{enumerate}
To avoid distracting peculiarities of the realizability topos this paper works with a generic \emph{$\Pi$-pretopos} with a \emph{natural number object}. That makes our simplicial homotopy predicative as well as constructive.

\subsection*{Intended model} The \emph{category of assemblies} which is the category of $\neg\neg$-separated objects of the effective topos up to equivalence, has several \emph{strongly} complete internal categories that are not posets, in particular the \emph{category of modest sets} \citep{MR1097022,MR2479466,MR1023803}. \emph{Strongly complete} means that the externalization of the internal category is a complete fibred category. The category of assemblies is not exact, but the \emph{ex/lex completion} \citep{MR1600009} preserves strongly complete internal categories. This exact completion is not the effective topos, but it is a kind of realizability topos and it is the intended ambient category $\ambient$ in this paper. 

The topos with a complete internal category that is not a poset is interesting, because complete internal categories in Grothendieck toposes are necessarily posets. This fact is traditionally attributed to Peter Freyd.

In the ex/lex completion \emph{modest fibrations} are fibrations whose underlying morphisms are families of modest sets or quotients of such families by modest families of equivalence relations--see \citep{MR1097022,MR1023803,MR2479466}. There is a \emph{universal modest fibration} (see definition \ref{universal modest fibration}) which is a \emph{univalent fibration} and hence a potential model of homotopy type theory.

\subsection*{Conclusion} There are definitions of fibrations, cofibrations and their acyclic counterparts (definitions \ref{model structure}) in $\Pi$-pretoposes that make the category of Kan complexes a \emph{model category} (theorem \ref{model category}). Certain \emph{universal fibrations} (definition \ref{universal modest fibration}) automatically live in the category of Kan complexes (theorem \ref{fibrant universe}).

\subsection*{Acknowledgments} 
I am grateful to the Warsaw Center of Mathematics and Computer Science for the opportunity to do the research leading to this paper. I am also grateful for discussions with Marek Zawadowski and the seminars on simplicial homotopy theory he organized during my stay at Warsaw University. Richard Garner, Peter LeFanu Lumsdaine and Thomas Streicher made invaluable comments on early drafts of this paper.

\section{Internal simplicial objects}
Internal simplicial homotopy has a notion of simplicial object of $\ambient$ which does rely on external $\ambient$-valued presheaves, but on internal structures.

\begin{definition} A \keyword{simplicial object} $X$ of $\ambient$ is a triple $\tuplet{\base X, \dim, \cdot}$ where $\base X$ is an object of $\ambient$, $\dim$ is a morphism $\base X\to \nno$. Here $\nno$ is the \keyword{natural number object} of $\ambient$. The operator $\cdot$ is has the following properties. Let $\Ar(\simCat)$ be the object of non decreasing functions between decidable, finite and inhabited initial segments of the natural numbers. The domain $\dom(\cdot)$ of $\cdot$ is the following object. 
\[ \set{\tuplet{x,\xi}\of \base X\times \Ar(\simCat)\middle| \dim(x) = \max(\cod(\xi))}\]
The codomain $\cod(\cdot)$ of $\cdot$ is $\base X$. The operator $\cdot$ satisfies the following equations.
\begin{align*}
\dim(x\cdot\xi) &= \max(\dom(\xi))\\
x\cdot \id &= x \\
x\cdot(\alpha\circ\beta)&=(x\cdot\alpha)\cdot\beta 
\end{align*}

A \emph{morphism of simplicial objects} $X\to Y$ is a morphism $f\of\base X\to\base Y$ which commutes with $\dim$ and $\cdot$. Together with simplicial objects, they form the category of simplicial objects and morphisms $\ambient\s$.

The category $\ambient\s$ is enriched over $\ambient$. The object of morphisms $X\to Y$ in $\ambient$ is $\nat(X,Y)$. It represents families of morphisms $X\to Y$, i.e. for each object $I$ of $\ambient$ there is a natural bijection between morphisms $I\to \nat(X,Y)$ and morphisms $I\times X\to Y$ which commute with $\dim$, $\cdot$ and their $I$-fold multiples.
\end{definition}

\begin{example} A simple kind of simplicial object is the \keyword{discrete} simplicial object. There is one for each object $I$ of $\ambient$: 
\[ I\disc = \tuplet{\nno\times I,\function{\tuplet{n,i}}n,\function{\tuplet{\tuplet{n,i},\phi}}i} \]
Note the $\function{x}{f(x)}$ notation for morphisms of $\ambient$ the paper uses.
\end{example}

This definition is consistent with the definition of initial presheaves in \citep{MR1300636}. The definition is valid because $\ambient$ is a $\Pi$-pretopos with a natural number object $\nno$. Decidable, finite and inhabited initial segments of the natural numbers have a classifier.
\[  \function{\tuplet{i,j}}j\of\set{\tuplet{i,j}\of\nno\times\nno\middle|i\leq j}\to \nno \]
The relation $\leq$ is decidable. Comprehension on decidable predicates defines subobjects in $\ambient$, because there classifiers of decidable subobjects $1\to \bool = 1+1$. Here $1$ is a terminal object and $+$ a binary coproduct.

Local Cartesian closure means that there is an object $A$ of morphisms between the initial segments. Local Cartesian closure also implies that there is a subobject of non decreasing morphisms $\Ar(\simCat)$.
\[ \product{f\of A,i,j\of\nno,x\of\set{y\of 1|i\leq j}}{\set{y\of 1|f(i)\leq f(j)}}\]
However, because $\ambient$ is also exact and extensive, it is a Heyting category, so the same subobject has a more familiar definition.
\[ \Ar(\simCat)=\set{f\of A|\forall i,j\of\nno.i\leq j \to f(i)\leq f(j)} \]

\begin{definition} The object $\Ar(\simCat)$ is the object of arrows of the internal \keyword{category of simplices} $\simCat$, whose object of objects $\Ob(\simCat)$ is $\nno$. Composition is defined as ordinary function composition.
\end{definition}

For completeness we recall the definition of $\Pi$-pretoposes below.

\begin{definition} 
That $\ambient$ is a $\Pi$-\keyword{pretopos} means all of the following.

\begin{enumerate}
\item The category $\ambient$ is \emph{locally Cartesian closed}. This means that for each object $X$ of $\ambient$ each slice $\ambient/X$ is Cartesian closed.
\item The category $\ambient$ is \emph{extensive}, which means that is has finite coproducts and that $\ambient/(X+Y)$ is equivalent to $(\ambient/X)\times(\ambient/Y)$ for each pair of objects $X$ and $Y$ of $\ambient$.
\item The category $\ambient$ is \emph{exact} in the sense of Barr. This means that all weak groupoids (see below) have coequalizers and that those coequalizers are stable under pullback.
\end{enumerate}

A \keyword{weak groupoid} consists of a pair of objects $X_0$, $X_1$ and morphisms $r\of X_0\to X_1$ and $s,t\of X_1\to X_0$ such that $s\circ r= t\circ r = \id_{X_0}$ such that for each $x,y,z\of X_0$ if two of $\tuplet{x,y}$, $\tuplet{x,z}$ and $\tuplet{y,z}$ are in the image of $\tuplet{s,t}\of X_1\to X_0\times X_0$, the third is too. A \emph{coequalizer} for a weak groupoid is a coequalizer of $s$ and $t$.
\end{definition}

\begin{example} Every topos is a $\Pi$-pretopos so $\ambient$ can be any topos with a natural number object like the topos of sets. \end{example}

This combination of properties implies the following.
\begin{enumerate}
\item The category $\ambient$ has all colimits and reindexing functors preserve them. 
\item The category $\ambient$ is a \emph{Heyting category}. 

Technically this means that for each object $X$ the poset of subobjects $\sub(X)$ is a Heyting algebra and that for each $f\of X\to Y$ the preimage map $f\ri\of\sub(Y)\to\sub(X)$ is a morphism of Heyting algebras. Moreover, for each morphism $f$, $f\ri$ has both adjoints--right adjoint $\forall_f$ and left adjoint $\exists_f$--and those adjoints satisfy the \emph{Beck-Chevalley condition} which says that quantification and substitution commute over pullback squares. If $f\circ g = h\circ k$ is a pullback square, then $f\ri\circ\exists_h=\exists_g\circ k\ri$.

Practically this means that $\ambient$ is a model for a constructive first order logic, hence that predicates define subobjects and that constructive theorems are valid in the internal language.
\end{enumerate}

\begin{remark} We assume that all of the structure on $\ambient$ comes from functors. So several functors between categories related to $\ambient$ have right adjoints that give $\ambient$ its limits and exponentials and some functors have left adjoints that give $\ambient$ colimits. Definitions in terms of universal properties only define objects up to isomorphism. Our assumption ensures that there are functors that hit all the necessary isomorphism classes, when we only specify a functor up to isomorphism.
\end{remark}

\section{Model Category}
This section defines a model structure on a category of $\ambient\s$.

\hide{ Finding  model structure on $\ambient$ is more difficult. Only the decidable-split system works, but is hard to extend to a system where weak equivalences satisfy 2-out-of-3. }

\begin{definition} An \keyword{model structure} on an $\cat A$-enriched category $\cat C$ consist of three sets of morphism: the \emph{fibrations} $F$, the \emph{cofibrations} $C$ and the \emph{weak equivalences} $W$ which have the following properties. 
\begin{itemize}
\item The set $W$ satisfies \emph{2-out-of-3} which means that for every pair of morphisms $f\of X\to Y$ and $g\of Y\to Z$ of $\cat C$, if two out of $f$, $g$ and $g\circ f$ are in $W$ then all three are. 
\item The pairs of sets $(C,F\cap W)$ and $(C\cap W,F)$ are \emph{enriched factorization systems} (see definition \ref{enriched factorization system}). 
\end{itemize}

Categories with such a structure are \emph{model categories}. Members of $C\cap W$ are called \keyword{acyclic cofibrations} and members of $F\cap W$ are called \keyword{acyclic fibrations}.
\end{definition}

This notion of model structure is the ordinary one, except for the kind of factorization systems the morphisms in it form.

\begin{definition}
An \keyword{enriched factorization system} is a pairs of set of morphisms $(L,R)$ of an enriched category with the following properties.
\begin{itemize}
\item A morphism belong to $L$ if and only if it has the \emph{left lifting property} with respect all to members of $R$.
\item A morphism belong to $R$ if and only if it has the \emph{right lifting property} with respect all to members of $L$.
\item Every morphism factors as a member of $R$ following a member of $L$.
\end{itemize}

A morphism $f\of X\to Y$ of an $\cat A$-enriched category $\cat C$ has the \emph{right lifting property} with respect to a morphism $g\of I\to J$--and $g$ has the \emph{left lifting property} with respect to $f$--if the morphism $\tuplet{f_!,g\ri} = \tuplet{\cat C(\id_J,f),\cat C(g,\id_X))}$ which is the factorization of the span $\cat C(\id_J,f)$, $\cat C(g,\id_X)$ through the pullback cone of $\cat C(\id_I,f)$ and $\cat C(g,\id_Y)$ is a \emph{split} epimorphism.
\[\xy
(34,20)*+{\cat C(I,X)}="top",(0,10)*+{\cat C(J,X)}="left",(24,10)*+{\bullet}="middle",(44,10)*+{\cat C(I,Y)}="right",(34,0)*+{\cat C(J,Y)}="bottom"
\ar^{\cat C(g,\id_X)} "left";"top" \ar@{.>}|(.6){\tuplet{f_!,g\ri}} "left";"middle" \ar_{\cat C(\id_J,f)} "left";"bottom" \ar "middle";"bottom"
\ar "middle";"top" \ar^(.6){\cat C(\id_I,f)} "top";"right" \ar_(.6){\cat C(g,\id_Y)} "bottom";"right"
\endxy\]
A section of $\tuplet{f_\bang,g\ri}$ is a \keyword{filler operator}.
\end{definition}

In $\ambient\s$ we make the following selections. 

\begin{definition} A morphism $f\of X\to Y$ of category $\ambient\s$ is a \keyword{fibration} if it has the right lifting property with respect to the family of all \emph{horn inclusions}, which is defined as follows. 

For each $n\in\N$, the $n$-\keyword{simplex} $\simplex[n]$ is the internal simplicial set where $\base{\simplex[n]}$ is the object of all morphisms $\phi$ in $\simCat$ such that $\cod(\phi) = n$, $\dim(\phi) = \dom(\phi)$ and $\phi\cdot \xi = \phi\circ \xi$. Here $\N$ is the \keyword{set of natural numbers} that exists outside of $\ambient$. For each $k\of[n]$ the \keyword{horn} $\horn_k[n]$ is the subobject of the non decreasing maps $[m]\to [n]$ that are not onto the set $[n]-\set{k}$.
\[ \base{\horn_k[n]} = \set{\phi\of\base{\simplex[n]}| \exists i\of[n]-\set k.\forall j\of \dom(\phi).\phi(i)\neq j } \]
The \keyword{horn inclusion} is the monomorphism $\horn_k[n]\to\simplex[n]$. The \emph{family of horn inclusions} is the sum of all horn inclusions $\coproduct{k\leq n\of\nno}{\horn_k[n]}\to \coproduct{k\leq n\of\nno}{\simplex[n]}$ that exists as a morphism of $\ambient\s$.

The lifting property that $f$ satisfies is defined in a slice category of $\ambient\s$. Let $I = \set{\tuplet{n,k}\of\nno\times\nno|k\leq n}$. The family of horn inclusions is a morphism in $\ambient\s/I\disc$ (see example \ref{discrete}). So is the multiple $I\disc\times f\of I\disc\times X\to I\disc\times Y$ of $f$. The morphism $f$ is a fibration if $I\disc\times f$ has the right lifting property with respect to $\coproduct{k\leq n\of\nno}{\horn_k[n]}\to \coproduct{k\leq n\of\nno}{\simplex[n]}$.
\end{definition}

\begin{example} If $\ambient$ is the topos of sets, then the morphisms that satisfy this definition are precisely the \emph{Kan fibrations}. In that case the axiom of choice implies a filler operator exists. \end{example}

\begin{definition} A monomorphism $f\of X\to Y$ of category $\ambient\s$ is a \keyword{cofibration} if the subobject of faces of $Y$ which are not in the image of $X$ is decidable. A \keyword{face} of $Y$ is a $y\of\base Y$ such that if $\phi\of [\dim(y)]\to [\dim(y)]$ in $\simCat$ and $y\cdot \phi = y$, then $\phi = \id_{\dim(y)}$. Elements of $\base Y$ which are not faces are called \keyword{degenerate}, so a face is a non degenerate element of $\base Y$. Hence $f$ is a cofibration if there is a morphism $\psi\of\base Y\to\bool$ such that $\psi(y)=\false$ if an only if $y$ is degenerate or equal to $f(x)$ for some $x\of \base X$.\end{definition}

\begin{example} In the category of simplicial sets all monomorphism are cofibrations. Decidability comes for free there.\end{example}

\begin{definition} A \keyword{weak equivalence} is a composition of an acyclic fibration following an acyclic cofibration. Here, an \emph{acyclic fibration} is a morphism that has the right lifting property with respect to all cofibrations; an \emph{acyclic cofibration} is a cofibration has the right lifting property with respect to all fibrations.
\end{definition}

The theorems in this paper should work in a $\Pi$-pretopos $\ambient$ that doesn't have infinite colimits. The small object argument uses such colimits to provide factorizations of morphisms and therefore we cannot use it here. To get a model structure we retreat to a subcategory of $\ambient\s$.

\begin{definition} A \keyword{complex} is a tuple $\tuplet{\base X,\dim,\cdot,\filler}$ where $\tuplet{\base X,\dim,\cdot}$ is a simplicial object and $\filler$ is a filler operator (see definition \ref{enriched factorization system}) for the unique morphism $\bang\of \tuplet{\base X,\dim,\cdot}\to 1\disc$. A morphism of complexes is simply a morphism of simplicial objects. The \emph{category of complexes} and morphisms of complexes is $\ambient\s\f$.
\end{definition}

\begin{example} Each discrete simplicial object has a canonical filler operator, which means that $\function{I\of\ambient}{I\disc}$ factors through the category of complexes. All simplices are complexes. Horns are not, however.
\end{example}

\begin{example} When $\ambient$ is the category of sets, then complexes are Kan complexes with a filler operator included. \end{example}

\begin{theorem}[Model category]
With fibrations, weak equivalences and cofibrations defined as above the $\ambient$-enriched category $\ambient\s\f$ is a model category.
\label{model category}
\end{theorem}

\begin{proof}
Lemma \ref{factorization system 1} shows that cofibrations and acyclic fibrations form a weak factorization system. Lemma \ref{factorization system 2} tells the same thing about fibrations and acyclic cofibrations. Lemma \ref{toot} demonstrates that if two of $f$, $g$ and $f\circ g$ are weak equivalences, then all three are. These three requirements define a model structure.
\end{proof}

\section{Cofibrations}
This section show that the cofibrations from definition \ref{cofibration} are part of an enriched factorization system (see definition \ref{enriched factorization system}) on $\ambient\s$--not just on the subcategory of complexes.

\begin{lemma} Cofibrations and acyclic fibrations form an enriched factorization system on $\ambient\s$. \label{factorization system 1}. \end{lemma}

\begin{proof} Proposition \ref{factor1} shows that every morphism factors as a \emph{contractible morphism} (see definition \ref{contractible}) following a cofibration. Acyclic fibrations are examples of contractible morphisms (see example \ref{acyclic means contractible}). Lemma \ref{Reedy} shows that contractible morphisms are acyclic fibrations. Hence contractible morphisms and acyclic fibrations are the same class of morphism and this class form the right class of a factorization system with cofibrations on the left (see definition \ref{enriched factorization system}).
\end{proof}

\begin{definition} A morphism $f\of X\to Y$ is \keyword{contractible} if it has the internal left lifting property with respect to the family of cycle inclusions, which is defined as follows.

The \emph{cycle} $\cycle[n]$ is the subobject of non surjective functions of $\simplex[n]$. The $\dim$ and $\cdot$ are the same and the bases are related as follows:
\[ \base(\cycle[n]) = \set{\phi\of\base(\simplex[n])| \exists i\of[n].\forall j\of \dom(\phi).\phi(i)\neq j } \]
The \keyword{cycle inclusion} is the monomorphism $\cycle[n]\to\simplex[n]$. The \emph{family of cycle inclusions} is the sum of all cycle inclusions $\coproduct{n\of\nno}{\cycle[n]}\to \coproduct{k\leq n\of\nno}{\simplex[n]}$ that exists as a morphism of $\ambient\s$. Note that $\coproduct{x\of y}{z(x)}$ denotes the \keyword{indexed coproduct} of the $y$-indexed family $z$.

Definition \ref{fibration} explains how lifting properties with respect to families work.
\end{definition}

\begin{example} Each cycle inclusion $\cycle[n]\to\simplex[n]$ is a cofibration because $\id_{[n]}\of[n]\to[n]$ is the unique face of $\simplex[n]$ not in $\cycle[n]$ and equility with $\id$ is decidable in $\simCat$. Therefore every acyclic fibration is a contractible morphism. \label{acyclic means contractible}\end{example}

\begin{proposition} Every morphism factors as a contractible morphism following a cofibration. \label{factor1} \end{proposition}

\begin{proof} \hide{The construction is based on the fact that $\simCat$ is a \emph{Reedy category}.} 
Let $\simCat^+$ be the subcategory of monomorphisms and $\simCat^-$ that of epimorphisms of $\simCat$ (see definition \ref{category of simplices}). For each morphism $\phi$ of $\simCat$ let $m(\phi)$ be the monic and $e(\phi)$ the epic factors.

The first step is to cover $Y$ with another simplicial set $LY$ where degeneracy is decidable.
Let $LY[n] = \coproduct{i\of[n]}{Y[i]\times\simCat^-(n,i)}$ and for $\phi\of [m]\to [n]$ of $\simCat$ and $\tuplet{\epsilon,y}\of LY[n]$ let the following equation hold.
\[\tuplet{\epsilon,y}\cdot \phi=\tuplet{e(\epsilon\circ \phi),y\cdot m(\epsilon\circ\phi)}\]
Let $l_Y\of LY\to Y$ satisfy $l_Y\tuplet{\epsilon,y}=y\cdot\epsilon$.

The second step glues simplices of $X$ and $LY$ together into new ones.
The pair $(f,l_Y)$ stand for the morphism $X+LY\to Y$ that satisfies $(f,l_Y)(x)=f(x)$ for $x\of X$ and $(f,l_Y)(y) = l_Y(y)$ for $y\of LY$. Let $Z[n]\subseteq \product{i\in[n]}{(X[i]+LY[i])^{\simCat^+(i,n)}}$ consist of elements $z$ which satisfy the following conditions for all $\alpha\of [j]\to [k]$ and $\beta\of [i]\to[j]$ in $\simCat^+$.
\begin{enumerate}
\item $(f,l_Y)(z(\alpha\circ\beta)) = (f,l_Y)(z(\alpha))\cdot\beta$;
\item if $z(\alpha)\of X$ then $z(\alpha\circ\beta)\of X$;
\item if $z(\alpha)=\tuplet{\epsilon,y}\of LY$ and $\epsilon\circ\beta$ is not monic, then $z(\alpha\circ\beta)\of LY$ and $z(\alpha\circ\beta)_0$ factors through $\epsilon\circ\beta$.
\end{enumerate}
For $\phi\of[m]\to [n]$ and $z\of Z[n]$, let $z\cdot\phi$ satisfy the following equation for all $\alpha$ of $\simCat^+$.
\[ (z\cdot\phi)(\alpha) = z(m(\phi\circ\alpha))\cdot e(\phi\circ \alpha) \]
By these definitions, $Z=(\coproduct{n\of\nno}{Z[n]},\function{\tuplet{n,z}}n,\cdot)$ is a simplicial object.

There is a morphism $g\of X\to Z$ which satisfies $g(x)(\alpha) = x\cdot\alpha$. It is a cofibration because a $z\of\base Z$ is nondegenerate and outside of the image of $g$ if and only if $z(\id_{[\dim(z)]}) = \tuplet{\id_{[\dim(z)]},y}$ for some $y\of Y[\dim(z)]$. The definition of $LY$ ensures that elements with this property are nondegenerate while condition 3 ensures that if $z(\id)$ is degenerate, then so is all of $z$.

There is a morphism $h\of Z\to Y$ which satisfies $h(z) = (f,l_Y)(z(\id_{[\dim(z)]}))$. This morphism is contractible thanks to the following filler construction. Let $a\of \cycle[n]\to Z$ and $b\of \simplex[n]\to Y$ satisfy $h\circ a = b\circ k$ where $k\of\cycle[n]\to\simplex[n]$ is the cycle inclusion. Let $c\of \simplex[n]\to Z$ satisfy $c(\phi)(\mu) = \tuplet{\phi,b(\mu)}$ if $\phi$ is an epimorphism and $c(\phi)(\mu) = a(\phi)(\mu)$ otherwise. Now $c\circ k = a$ and $h\circ c = b$ so $c$ is a filler. There is a morphism in $\ambient\s$ that maps each commutative square $h\circ a = b\circ k$ to this filler $c$ because $\ambient\s$ is locally Cartesian closed and being an epimorphism is a decidable property of morphisms of $\simCat$.
\end{proof}\hide{ Codomains of contractibles may have a distance function, that gives a lower bound to the number of compositions required to to reach it. I don't have time to work that out and check it right now.}

\begin{lemma} Contractible morphisms are acyclic fibrations.\label{Reedy}\end{lemma}

\begin{proof} The class of morphisms that have the left lifting property with respect to contractible morphism is saturated (see lemma \ref{saturation}). The rest proof shows that all cofibrations are in this class. That implies that all contractible morphisms have the right lifting property with respect to all cofibrations and therefore are acyclic fibrations.

Suppose $f\of X\to Y$ is a cofibration and $S$ is the family of faces outside the image of $f$. For each $i\of\nno$ let $Y_i$ be the union of $X$ with all faces of $Y$ of dimension strictly smaller than $i$. In particular $Y_0=X$.

Every $y\of\base(Y)$ is in $Y_i$ for some $i\leq \dim(y)$. If a $y\of\base Y$ is degenerate, then one can search the monomorphisms $[n]\to\dim(y)$ in $\simCat$ for the greatest $\mu$ such that $y\cdot\mu$ is a face. For this reason, the inclusion $Y_i\to\nno\ri Y$ is a transfinite composition (see definition \ref{transfinite composition}).

For each $i\of\nno$ let $S_i$ be the object of $n$-dimensional faces in $S$. Each $s\of S_i$ induces a monomorphism $s'\of\simplex[i]\to Y_{i+1}$. Since $Y_{i}$ has all simplices of dimension $i-1$, the intersection of $Y_i$ and $s'$ is $\cycle[i]$. Thus the inclusion $Y_i\to Y_{i+1}$ is the pushout of $S_{i+1}$ copies of $\cycle[i]\to\simplex[i]$.
\end{proof}

\begin{definition} Let $\cat A$ be a $\Pi$-pretopos with a natural number object $\nno$. The successor morphism $s\of\nno\to\nno$ induces an endofunctor $s\ri$ of $\cat A/\nno$. A \keyword{cochain} is a morphism $f\of X\to s\ri(X)$ in $\cat A/\nno$. A morphism of cochains $(X,f)\to(Y,g)$ is a morphism $h\of X\to Y$ such that $s\ri(h)\circ f = g\circ h$.

For every object $Y$ of $\cat A$ there is a constant family $\nno\ri(Y)$ and a canonical isomorphism $\nno\ri(Y)\simeq s\ri(\nno\ri(Y))$ which makes $\nno\ri(Y)$ a cochain. The \keyword{transfinite composition} of a cochain $(X,g)$ is a cochain morphism $\traco(g)\of (X,f)\to\nno\ri Y$ such that for every other morphism $k\of (X,f)\to\nno\ri Z$ there is a unique morphism $l\of Z\to Y$ such that $s\ri(l)\circ \traco(g)=k$.
\end{definition}

\begin{lemma} The class of morphisms $L$ that has the left lifting property with respect to fibrations is closed under pushouts, coproducts indexed over objects of $\ambient$ and transfinite compositions. \label{saturation}\end{lemma}

\begin{proof}
In each case the construction induces an operation on the split epimorphism in the diagram of the lifting property. These operations happen to preserve split epimorphisms.

Suppose $h\of I'\to J'$ is a pushout of $g\of I\to J$ where $g$ has the left lifting property with respect to $f\of X\to Y$. Because the functors $\nat(-,X)$ and $\nat(-,Y)\of(\ambient\s)\dual\to\ambient$ send pushouts to pullbacks, $\tuplet{f_!,h\ri}$ is a pullback of $\tuplet{f_!,g\ri}$ and the former is a split epimorphism because the latter is.

Suppose $h\of(\product{i\of I}{\nat(X(i),Y(i))})$ represents a family of morphisms of that share the left lifting property. That means $\tuplet{(I\disc\ri(f))_!,h\ri}$ is a split epimorphism and therefore so is its transpose $\tuplet{f_!,\amalg_I(h)\ri}$, where $\amalg_I(h)\of\nat(\coproduct{i\of I}{X(i)},\coproduct{i\of I}{Y(i)})$ is the indexed coproduct of the family $h$.

Finally, suppose that a cochain $(Z,h)$ has the left lifting property. That means $\tuplet{\nno\disc\ri(f)_!,h\ri}$ is a split epimorphism. For the transfinite composition $\traco(h)$, $\tuplet{f_!,\traco(h)\ri}$ is therefore also split.
\end{proof}

\section{Weak Equivalences}
This section demonstrates that the weak equivalences of definition \ref{weak equivalence} satisfy the two out of three property (see definition \ref{model structure}).

\begin{lemma}[2-out-of-3] Let $f\of X\to Y$ and $g\of Y\to Z$ be morphisms of $\ambient\s\f$. If any two of $f,g$ or $g\circ f$ are weak equivalences, then all three are. \label{toot}\end{lemma}

\begin{proof} Weak equivalences are closed under composition by lemma \ref{composition of weak equivalences}.

Let $g$ and $g\circ f$ be arbitrary weak equivalences. The morphism $f$ factors as an acyclic fibration $h\of W\to Y$ following a cofibration $k\of X\to W$ by proposition \ref{factor2}. Because weak equivalences are closed under composition (see \ref{composition of weak equivalences}), $g\circ h$ is a weak equivalence. The morphism $k$ is acyclic for the following reasons. Factor both $g\circ f$ and $g\circ h$ as acyclic fibrations following acyclic cofibrations, so $g\circ f = a\circ b$ and $g\circ h = c\circ d$. The lifting properties induce a morphism $l$ such that $l\circ b = d\circ k$ and $c\circ l = a$. Lemma \ref{shared retract 2} says that $l$ is a weak equivalence because $a$ and $c$ are acyclic fibrations. 
Because of closure under composition, the morphism $l\circ b = d\circ k$ is both a weak equivalence and a cofibration and hence an acyclic cofibration. Since $d$ and $d\circ k$ are acyclic cofibrations, so is $k$ by lemma \ref{left cancellation}.
\[\xymatrix{
X\ar[r]_k\ar@/^2ex/[rr]^f\ar[d]_b & W\ar[r]_h\ar[d]^d & Y\ar[d]^g \\
\bullet\ar@/_2ex/[rr]_a\ar@{.>}[r]^l & \bullet\ar[r]^c & Z
}\]
Since $f = h\circ k$, $f$ is a weak equivalence.

The case where $f$ and $g\circ f$ are weak equivalences is dual to the case above and the reasoning is the same. Acyclic cofibrations satisfy lemma \ref{shared retract} where acyclic fibrations satisfy lemma \ref{shared retract 2}. Acyclic fibrations satisfy lemma \ref{right cancellation} where acyclic cofibrations satisfy lemma \ref{left cancellation}.

This means that weak equivalences indeed satisfy 2-out-of-3.
\end{proof}

\begin{lemma} Weak equivalences are closed under composition. \label{composition of weak equivalences}\end{lemma}

\begin{proof} Compositions of acyclic fibrations are acyclic fibrations and the same holds for acyclic cofibrations.
All compositions of weak equivalences are weak equivalences, if $g\circ f$ factors as an acyclic fibration following an acyclic cofibration for each acyclic cofibration $g$ and acyclic fibration $f$.

By proposition \ref{factor1} $g\circ f=h\circ k$ for some acyclic fibration $h\of W\to Z$ and a cofibration $k\of X\to W$. Let $g'$ be the left inverse of $g$. Since $f \circ \id = g'\circ g\circ f= (g'\circ h)\circ k$ there is a morphism $k'$ such that $f\circ k' = g'\circ h$ and $k'\circ k = \id$, so $k$ has its own left inverse.
\[\xymatrix{
X\ar[d]_f \ar[r]_{k} \ar@/^2ex/[rr]^{\id} & W\ar[d]^h \ar[r]_{k'} & X\ar[d]^f\\
Y \ar[r]^{g} \ar@/_2ex/[rr]_{\id} & Z \ar[r]^{g'} & Y
}\]

Let $\phi$ be a homotopy between $\id_Z$ and $g\circ g'$ such that $\phi\circ (\id\times g)=g\circ\pi_1$.

There is a homotopy $\chi$ between $\id_W$ and $k\circ k'$ by lemma \ref{triple lift}.
\begin{align*}
(\id, k\circ k')\circ (k+k) &= (k,k) = k\circ \pi_1\circ(c\times\id_X)\\
h\circ (k \circ \pi_1) &= (\phi\circ (\id_W\times h))\circ (\id_{\simplex[1]} \times k)\\
h\circ (\id, k\circ k') &= (\phi\circ (\id_W\times h))\circ (c\times \id_W)
\end{align*}
\[\xymatrix{
X+X\ar[d]_{c\times \id}\ar[r]^{k+k} & W+W\ar[d]_(.3){c\times \id}\ar[r]^(.6){(\id,k\circ k')} & W\ar[d]^h\\
\simplex[1]\times X\ar[r]_{\id\times k} \ar[urr]^(.3){k\circ \pi_1}  & \simplex[1]\times W\ar[r]_(.6){\phi\circ (\id\times h)} \ar@{.>}[ur]_\chi & Z
}\]
Because the homotopy satisfies $\chi\circ(\id_{\simplex[1]}\times k) = k\circ \pi_1$, $k$ is a weakly invertible cofibration by definition \ref{weakly invertible} and an acyclic cofibration by lemma \ref{acyclic have lifting}.
\end{proof}

\begin{lemma}[Triple lifting property] Let $f\of A\to B$ and $g\of C\to D$ be cofibrations and let $h\of X\to Y$ be a fibration. Let $a\of A\times D\to X$, $b\of B\times C\to X$ and $c\of B \times D\to Y$ satisfy $a\circ(\id_A\times g) = b\circ(f\times \id_C)$, $h\circ a=c\circ(f\times \id_D)$ and $h\circ b=c\circ(\id_B\times g)$. If one of $f$, $g$ or $h$ is acyclic, then there is a $d\of B\times D\to X$ such that $d\circ(f\times\id_D)=a$, $d\circ(\id_B\times g)=b$ and $h\circ d = c$.
\[\xy
(0,20)*+{A\times C}="AC",(25,20)*+{B\times C}="BC",(40,20)*+{X}="X",
(0,0)*+{A\times D}="AD",(25,0)*+{B\times D}="BD",(40,0)*+{Y}="Y"
\ar^{f\times\id} "AC";"BC"
\ar_{\id\times g} "AC";"AD"
\ar_{f\times\id} "AD";"BD"
\ar^{a} "AD";"X"
\ar|(.6){\id\times g} "BC";"BD"
\ar^(.6){b} "BC";"X"
\ar@{.>}_{d} "BD";"X"
\ar_(.6){c} "BD";"Y"
\ar^{h} "X";"Y"
\endxy\]
\label{triple lift}
\end{lemma}

\begin{proof} The general case reduces to the cases where $f$ and $g$ are cycle or horn inclusions, because the lemma is equivalent to the statement that the following maps are (acyclic) fibrations and because acyclic cofibrations are contractible (see definition \ref{contractible}).
\begin{align*}
\tuplet{f\to h}&=\tuplet{h^{\id_B},\id_X^f}\of X^B \to Y^B\times_{Y^A} X^A\\
\tuplet{g\to h}&=\tuplet{h^{\id_D},\id_X^g}\of X^D \to Y^D\times_{Y^C} X^C
\end{align*}
In the case where $h$ is an acyclic fibration the morphism $\tuplet{g\to h}$ is an acyclic fibration if is has the lifting property for the family of cycle inclusions $k_i\of \cycle[i]\to\simplex[i]$ for $i\of\nno$, because acyclic fibrations are contractible morphisms (see lemma \ref{Reedy}). In turn $\tuplet{k_n\to h}$ are acyclic cofibrations if the triple lifting property holds in the cases where $a=k_i$ and $b=k_j$ for all $i,j\of\nno$.

For the cases where $f$ or $g$ are acyclic the same reduction takes us to products of cycle and horn inclusions. If $f$ is acyclic, the problem of lifting $f$ against $g\to h$ reduces to the problem of lifting horn inclusions. The problem of lifting a possibly not acyclic $g$ against $k\to h$, where $k$ is a horn inclusion, reduces to the problem of lifting cycle inclusions by lemma \ref{Reedy}.

In the cases where $f$ and $g$ are horn or cycle inclusions it is easy to prove that their pushout products $g\pp f$ are compositions of pushouts of sums of horns and cycles and therefore have the left lifting properties with respect to (acyclic) fibrations (see lemma \ref{saturation}).
\[\xy
(0,10)*+{W\times Y}="left",(20,10)*+{\bullet}="middle",(10,20)*+{X\times Y}="top",(10,0)*+{W\times Z}="bottom",(40,10)*+{X\times Z}="right"
\ar^(.4){f\times \id} "left";"top" \ar_{\id\times g} "left";"bottom" \ar "top";"middle" \ar "bottom";"middle"
\ar@{.>}|(.3){g\pp f} "middle";"right" \ar^{\id\times g} "top";"right" \ar_{f\times\id} "bottom";"right"
\endxy\]
Therefore the lemma holds for every pair of cofibrations $f$, $g$ as long as one of $f$, $g$ and $h$ is acyclic.
\end{proof}

\begin{lemma} If $f\of X\to Y$, $g\of Y\to Z$ and if $g$ and $g\circ f$ are acyclic cofibrations, then $f$ is an acyclic cofibration \label{left cancellation} \end{lemma}

\begin{proof} Let $k\of A\to B$ be a fibration and let $a\of X\to A$ and $b\of Y\to B$ satisfy $k\circ a=b\circ f$. Because $B$ is a complex, there is a $b'\of Z\to B$ such that $b'\circ g = b$. Lifting properties also imply that there is an $a'\of Z\to A$ such that $a'\circ g\circ f = a$ and $k\circ a'= b'$. So $b'\circ g$ is a filler for $k\circ a=b\circ f$. By generalization this construction lifts $f$ against all fibrations and that makes $f$ an acyclic cofibration.
\[\xy
(0,28)*+{X}="x",(14\halfrootthree,21)*+{A}="a",(0,14)*+{Y}="y",(14\halfrootthree,7)*+{B}="b",(0,0)*+{Z}="z"
\ar^a "x";"a" \ar_f "x";"y" \ar^k "a";"b" \ar^(.25)b "y";"b" \ar_g "y";"z"
\ar@{.>}_{b'} "z";"b" \ar@{.>}^(.75){a'} "z";"a"
\endxy\]
\end{proof}

\begin{lemma} If $f\of X\to Y$, $g\of Y\to Z$ and if $f$ and $g\circ f$ are acyclic fibrations, then $g$ is an acyclic fibration. \label{right cancellation}\end{lemma}

\begin{proof} This is nearly the dual of lemma \ref{left cancellation} and dual reasoning gives $g$ the right lifting property for all cycle inclusions  which makes $g$ contractible (see definition \ref{contractible}) and an acyclic cofibration by lemma \ref{Reedy}.

\[\xy
(14\halfrootthree,28)*+{X}="x", (0,21)*+{\cycle[n]}="v", (14\halfrootthree,14)*+{Y}="y", (0,7)*+{\simplex[n]}="w", (14\halfrootthree,0)*+{Z}="z"
\ar@{.>}^{a'} "v";"x" \ar_(.75)a "v";"y" \ar_k "v";"w"
\ar@{.>}_(.25){b'} "w";"x" \ar_b "w";"z"
\ar "x";"y" \ar "y";"z"
\endxy\]

The reason the dual reasoning works is that for the cycles $\cycle[n]$ from definition \ref{contractible} the unique morphism $\bang\of 0\to \cycle[n]$ is a cofibration. Here $0=0\disc$ is the initial object of $\ambient\s$. A member of $\base(\cycle[n])$ is a face if it is an identity morphism, and this is a decidable property of morphism in $\simCat$. No faces are in the image of $\bang$.
\end{proof}

\begin{lemma} If $f\of X\to Y$, $g\of Y\to Z$ and if $g$ and $g\circ f$ are acyclic fibrations, then $f$ is a weak equivalence.\label{shared retract 2} \end{lemma}

\begin{proof} By proposition \ref{factor1}, $f$ factors as an acyclic fibration $h\of W\to Y$ following a cofibration $k\of X\to W$. Because $(g\circ f)\circ \id = (g\circ h)\circ k$ and $g\circ f$ is an acyclic fibration, $k$ has a left inverse $k'\of X\to W$ which satisfies $g\circ f\circ k' = g\circ h$. 
\[\xymatrix{
X\ar[d]_k \ar[r]^\id & X\ar[d]^{g\circ f}\\
W\ar[r]_{g\circ h} \ar@{.>}[ur]^{k'} & Z
}\]

Let $c\of 1+1\to\simplex[1]$ be the same cycle as above. There is homotopy $\phi$ between $\id_W$ and $k\circ k'$ by lemma \ref{triple lift} and the following equations.
\begin{align*}
(\id,k\circ k')\circ(k+k) &= (k,k) = (k\circ\pi_1)\circ (c\times\id_X)\\
(g\circ h)\circ (k\circ \pi_1) &= (g\circ h\circ \pi_1)\circ(\id\times k)\\
(g\circ h)\circ (\id,k\circ k') &= (g\circ h\circ \pi_1)\circ(c\times\id_W)
\end{align*}
\[\xymatrix{
X+X \ar[d]_{c\times \id}\ar[r]^{k+k} & W+W\ar[d]_(.3){c\times\id}\ar[r]^(.6){(\id,k\circ k')} & W\ar[d]^{g\circ h}\\
\simplex[1]\times X \ar[r]_{\id\times k}\ar[urr]^(.3){k\circ\pi_1}& \simplex[1]\times W\ar[r]_{g\circ h\circ \pi_1}\ar@{.>}[ur]_\phi & Z
}\]
Because $\phi\circ (\id_{\simplex[1]}\times k)=k\circ \pi_1$, definition \ref{weakly invertible} says that $k$ is a weakly invertible cofibration and lemma \ref{acyclic have lifting} says that $k$ is an acyclic cofibration. Therefore $f$ is a weak equivalence.\end{proof}

\begin{lemma} If $f\of X\to Y$, $g\of Y\to Z$ and if $f$ and $g\circ f$ are acyclic cofibrations, then $g$ is a weak equivalence.\label{shared retract} \end{lemma}

\begin{proof} \hide{The morphism $g$ factors as a fibration $h\of W\to Z$ following an acyclic cofibration $k\of Y\to W$ by lemma \ref{factor2}. Lifting properties give $h$ a right inverse $h'$.
\[\xymatrix{
X\ar[r]^{k\circ f}\ar[d]_{g\circ f} & W\ar[d]^h\\
Z\ar[r]_\id \ar[ur]^{h'} & Z
}\]
Lemma \ref{triple lift} provides a homotopy $\phi$ between $\id_W$ and $h'\circ h$, because the following equations hold.
\begin{align*}
(\id,h'\circ h)\circ(k\circ f+k\circ f)&= (k\circ f,k\circ f) = (k\circ f\circ \pi_1)\circ(c\times \id_Y)\\
h\circ(k\circ f\circ\pi_1) &= g\circ f\circ\pi_1 = (h\circ\pi_1)\circ(\id_{\simplex[1]}\times (k\circ f))\\
h\circ(\id_W,h'\circ h) &= (h,h) = (h\circ\pi_1)\circ(c\times\id_W)
\end{align*}

\[\xymatrix{
X+X\ar[r]^{k\circ f+k\circ f}\ar[d]_{c\times \id} & W+W\ar[r]^(.6){(\id,h'\circ h)}\ar[d]_(.3){c\times\id} & W\ar[d]^h\\
\simplex[1]\times X\ar[r]_{\id\times (k\circ f)}\ar[urr]^(.3){k\circ f\circ\pi_1} & \simplex[1]\times W\ar[r]_(.6){h\circ \pi_1}\ar@{.>}[ur]_{\phi} & Y
}\]}

This lemma is dual to lemma \ref{shared retract 2}. The morphism $g$ factors as a fibration $h\of W\to Z$ following an acyclic cofibration $k\of Y\to W$ by lemma \ref{factor2}. The dual of the proof of lemma \ref{shared retract 2} doesn't show that $h$ is an acyclic fibration directly, but it does show that $h$ has a right inverse $h'\of Z\to W$ and that there is a homotopy $\phi\of\simplex[1]\times W\to W$ between $\id_W$ and $h'\circ h$.

Let $a\of I\to J$ be an arbitrary cofibration and let $i\of I\to W$ and $j\of J\to Z$ satisfy $j\circ a=h\circ i$. Lemma \ref{triple lift} deforms $h'\circ j$ into a filler. Let $c_i$ be the morphisms $1\to \simplex[0]$.
\begin{align*}
\phi\circ(\id\times i)\circ (c_1\times \id) &= h'\circ h\circ i = h'\circ j\circ a \\
h\circ h' \circ j &= j = j\circ \pi_1\circ(c_1\times\id) \\
j\circ \pi_1\circ (\id\times a) &= j\circ a\circ \pi_1 = h\circ i \circ \pi_1\\
h\circ \phi\circ(\id\times i) &= h \circ \pi_1 \circ (\id\times i) = h\circ i \circ \pi_1
\end{align*}

\[\xymatrix{
I\ar[r]^{a}\ar[d]_{c_1\times \id} & J\ar[r]^(.6){h'\circ j}\ar[d]_(.3){c_1\times\id} & W\ar[d]^h\\
\simplex[1]\times I\ar[r]_{\id\times a}\ar[urr]^(.3){\phi\circ(\id\times i)} & \simplex[1]\times J\ar[r]_(.6){j\circ \pi_1}\ar@{.>}[ur]_{\psi} & Z
}\]
The filler is $\psi\circ c_0$. By generalization, $h$ is an acyclic fibration and $f$ is a weak equivalence.
\end{proof}

\section{Fibrations}
This section show that fibrations are part of a factorization system as well.

\begin{lemma} Acyclic cofibrations and fibrations form an enriched factorization system on $\ambient\s\f$. \label{factorization system 2} \end{lemma}

\begin{proof} Lemma \ref{acyclic have lifting} shows that weakly invertible cofibrations (see definition \ref{weakly invertible}) are acyclic cofibrations. Lemma \ref{lifting is acyclic} shows the converse. Proposition \ref{factor2} shows that every morphism factors as a fibration following a weakly invertible cofibration.
\end{proof}

\begin{definition} In the subcategory $\ambient\s\f$ of $\ambient\s$ a cofibration $f\of X\to Y$ is \keyword{weakly invertible} if there is a $g\of Y\to X$ such that $g\circ f = \id_X$ and an $h\of \simplex[1]\times Y\to Y$ such that $h\circ k = (\id_Y,f\circ g)$ if $k\of 1+1\to\simplex[1]$ is the cycle inclusion and 
$h\circ (\id_{\simplex[1]}\times f) = f\circ\pi_1$.
\[\xymatrix{
Y+Y\ar[d]_{k\times \id} \ar[dr]^{(\id,f\circ g)}\\
\simplex[1]\times Y \ar[r]_(.6)h & Y\\
\simplex[1]\times X \ar[r]_(.6){\pi_1}\ar[u]^{\id\times f} & X\ar[u]_f\\
}\]
\end{definition}

\begin{lemma} Weakly invertible cofibrations are acyclic cofibrations. \label{acyclic have lifting}\end{lemma}

\begin{proof} Let $f\of X\to Y$ be a weakly invertible cofibration, with inverse $g\of Y\to X$ and homotopy $h\of\simplex[1]\times Y\to Y$. Let $k\of A\to B$ be a fibration and let $a\of X\to A$ and $b\of Y\to B$ satisfy $k\circ a=b\circ f$. The composition $a\circ g\of Y\to A$ is a filler up to homotopy, which can be transported along the homotopy $h$ to become a proper filler. Let $d_0$ be the first horn inclusion $1\to \simplex[1]$. By lemma \ref{triple lift}, there is a $d\of \simplex[1]\times Y\to A$ such that $k\circ c=b\circ h$, $c\circ (\id_{\simplex[1]}\times f)= a\circ \pi_1$ and $c\circ (d_0\times \id_Y) = a\circ g$.
Lemma \ref{triple lift} applies because the following equations hold.

\begin{align*} 
	(a\circ \pi_1)\circ (\simplex(\delta^1_1)\times\id_X) &= a = (a\circ g)\circ f\\
	(b\circ h)\circ (\simplex(\delta^1_1)\times\id_Y) &= b\circ f\circ g = k\circ (a\circ g)\\
	(b\circ h)\circ (\id_{\simplex[1]}\times f) &= b\circ f\circ \pi_1 = k\circ (a\circ \pi_1)
\end{align*}
\[\xymatrix{
X\ar[d]_f \ar[rr]^(.4){d_0\times \id} && \simplex[1]\times X \ar[d]|(.6){\id\times f}\ar[r]^(.6){a\circ \pi_1} & A\ar[d]^k\\
Y \ar[rr]_(.4){d_0\times \id}\ar[urrr]^(.3){a\circ g} && \simplex[1]\times Y\ar[r]_(.6){b\circ h}\ar@{.>}[ur]_c & B
}\]

Let $d_1$ be the other horn inclusion $1\to \simplex[1]$. Let $e=c\circ(d_1\times \id_Y)$, so $k\circ e=b\circ h\circ (d_1\times \id_Y) = b$ and $e\circ f = c\circ (\id_{\simplex[1]}\times f)\circ (d_1\times \id_X) = a$. Then $e$ is a filler for $b\circ f=k\circ a$. By generalization $f$ is a weakly invertible cofibration.
\end{proof}

\begin{lemma} All acyclic cofibrations in $\ambient\s\f$ are weakly invertible.\label{lifting is acyclic} \end{lemma}

\begin{proof} An acyclic fibration is a fibration, because the right lifting property for cofibrations implies the right lifting property for horn inclusions. Therefore every acyclic cofibration is a cofibration.

Let $f\of X\to Y$ be a morphism with the left lifting property for all fibrations in $\ambient\s\f$. There is a morphism $g\of Y\to X$ such that $g\circ f = \id_X$ because $X$ is fibrant. 
\[ \xymatrix{
X\ar[d]_f \ar[r]^\id & X\ar[d]^\bang\\
Y\ar[r]_\bang \ar@{.>}[ur]_h & 1
}\]
There is a morphism $h\of \simplex[1]\times Y\to Y$ such that $h\circ k = (\id_Y,f\circ g)$ if $k\of 1+1\to\simplex[1]$ is the cycle inclusion and $h\circ (\id_{\simplex[1]}\times f) = f\circ\pi_1$ because $Y$ is a complex and because of lemma \ref{triple lift}.
\[ \xymatrix{
X+X\ar[d]_{f+f}\ar[r]^{k\times\id} & \simplex[1]\times X \ar[r]^(.6){f\circ \pi_0}\ar[d] & Y\ar[d]^\bang\\
Y+Y\ar[r]_{k\times\id} \ar[urr]^(.3){(\id,g)} & \simplex[1]\times Y \ar[r]_\bang\ar@{.>}[ur]_h & 1
}\]
Therefore $f$ is a weakly invertible cofibration by definition \ref{weakly invertible}.
\end{proof}

\begin{proposition} Every morphism $f\of X\to Y$ of $\ambient\s\f$ factors as a fibration following a weakly invertible cofibration. \label{factor2} \end{proposition}

\begin{proof} There is a simple factorization $f = p_0\circ r$ where $p_0\of Y/f\to Y$ and $r\of X\to Y/f$ and where 
\begin{align*}
Y/f &= \set{(x,p)\of X\times Y^{\simplex[1]}| f(x) = p(\function{*}{1}) }\\
p_0\tuplet{x,p} &= p(\function{*}{0})\\
r(x) &= \tuplet{x,\function{*}{f(x)}}
\end{align*}

The morphism $r$ is not necessarily a cofibration, but factors as an acyclic fibration $g\of W\to Y/f$ following a cofibration $h\of X\to W$ by proposition \ref{factor1}. The morphism $h$ is a weakly invertible cofibration and the composition $p_0\circ g$ is a fibration for the following reasons.

There is a morphism $p_1\of Y/f\to X$ that satisfies $p_1\tuplet{x,p}=x$. Let $h'=p_1\circ h$, so $h'\circ h = \id_X$.
There is a homotopy $\phi$ between $\id_{W}$ and $h\circ h'$ by lemma \ref{triple lift} and the following equations.
\begin{align*}
(\id_W,h\circ h')\circ(h+h)&=(h,h)= (h\circ\pi_1)\circ (c\times \id_X)\\
h'\circ h\circ \pi_1&=\pi_1=(h'\circ\pi_1)\circ(\id_{\simplex[1]}\times h)\\
h'\circ(\id_W,h\circ h')&=(h',h')= (h'\circ\pi_1)\circ (c\times \id_W)
\end{align*}
\[\xymatrix{
X+X\ar[d]_{c\times \id}\ar[r]^{h+h} & W+W\ar[d]_(.3){c\times \id}\ar[r]^(.6){(\id,h\circ h')} & W\ar[d]^{h'}\\
\simplex[1]\times X\ar[r]_{\id\times h} \ar[urr]^(.3){h\circ\pi_1}  & \simplex[1]\times W\ar[r]_(.6){h'\circ\pi_1} \ar@{.>}[ur]_\phi & X
}\]
Because $\phi\circ(\id_{\simplex[1]\times h}) = h\circ \pi_1$, the cofibration $h$ is weakly invertible by definition \ref{weakly invertible}. 

Since fibrations are closed under composition and since $g$ is an acyclic fibration, $p_0\circ g$ is a fibration if $p_0$ is. 

Let $d\of A\to B$ be a weakly invertible cofibration and let $a\of A\to Y/f$ and $b\of B\to Y$ satisfy $p_0\circ a = b\circ d$. There is a morphism $c_0\of B\to X$ such that $c_0\circ d = p_1\circ a$ because $X$ is fibrant. Let $e_0$ be one of the horn inclusions $1\to \simplex[1]$. There is a morphism $c_1\of\simplex[1]\times B\to Y$ such that $c_1\circ{(\id_{\simplex[1]}\times d)}$ is the transpose of $p_0\circ a$ and $c_0\circ (e_0\times \id_B) = f\circ c_1$ because $Y$ is fibrant and because of lemma \ref{triple lift}.
\[\xy
(0,20)*+{A}="AC",(25,20)*+{\simplex[1]\times A}="BC",(50,20)*+{Y}="Y",
(0,0)*+{B}="AD",(25,0)*+{\simplex[1]\times B}="BD",(50,0)*+{1}="one"
\ar^(.4){e_0\times\id} "AC";"BC"
\ar_{d} "AC";"AD"
\ar_(.4){e_0\times\id} "AD";"BD"
\ar^(.3){f\circ c_0} "AD";"Y"
\ar|(.3){\id\times d} "BC";"BD"
\ar^(.6){a^t} "BC";"Y"
\ar@{.>}_{c_1} "BD";"Y"
\ar_(.6){\bang} "BD";"one"
\ar^{\bang} "Y";"one"
\endxy\]
The transpose $c_1^t\of B\to Y^{\simplex[1]}$ of $c_1$ and $c_0$ together factor as a morphism $c\of B \to Y/f$ which is a filler for $p_0\circ a = b\circ d$. By generalization, $p_0$ has the right lifting property with respect to all acyclic cofibrations and therefore is a fibration.
\end{proof}

\section{Descent}
This purpose of this section is to replace arguments based on \emph{minimal fibrations} in simplicial homotopy theory, in particular those that are related to homotopy type theory. We assume that there is a class $M$ of \emph{modest fibrations} in $\ambient\s$ with the following properties.
\begin{itemize}
\item $M$ is closed under pullbacks along arbitrary morphisms.
\item $M$ is closed under composition.
\item $M$ is closed under fibred exponentiation. This means that if $f:X\to Z$ belongs to $M$ and $g:Y\to Z$ is an arbitrary morphism, the fibred exponential $f^g_Z\of(\product{z\of Z}{X_z^{Y_z}})\to Z$ is in $M$ too.

This requirement may give us more limits than we need, but the intended model of modest morphisms in $\Asm\ex$ satisfies it.

\item $M$ contains all \emph{regular} monomorphisms--a monomorphism is \emph{regular} if it is an equalizer.
\end{itemize} 

The interesting case is where $M$ has the following structure in addition to the properties above.

\begin{definition} A \keyword{universal modest fibration} is a modest fibration $\uni\of U\to V$ such that every modest fibration $f\of X\to Y$ is the pullback of $\uni$ along a unique morphism $\chi(f)\of Y\to V$.
\end{definition}

The following fact makes the universal modest fibration a potential model of homotopy type theory.

\begin{theorem} If $\uni\of U\to V$ is a universal modest fibration, then $V$ is fibrant. \label{fibrant universe} \end{theorem}

\begin{proof} For each horn inclusion $h\of \horn_k[n]\to\simplex[n]$ and each pair of $v\of \horn_k[n]\to V$, the modest fibration $v\ri(\uni)\of \horn_k[n]\times_V U \to \horn_k[n]$ descends along $h$ to form a modest fibration $D(v\ri(\uni))\of \bullet\to\simplex[n]$ by lemma \ref{descent}. There is a $\chi(D(v\ri(\uni)))\of\simplex[n]\to V$ and $\chi(h_*(v\ri(\uni)))\circ h$ equals $v$ by definition \ref{universal modest fibration}.
\end{proof}

\begin{lemma}[Descent]
For each horn inclusion $h\of \horn_k[n]\to\simplex$ and each fibration $f\of X\to \horn_k[n]$ there is a fibration $Df\of Y\to\simplex[n]$ such that $f$ is the pullback of $Df$ along $h$.
\label{descent}
\end{lemma}

\begin{proof} There is no horn inclusion for $n=0$. In the case $n=1$, there are two maps $1\to\simplex[1]$. In both case $Df=\pi_0\of \simplex[1]\times X \to\simplex[1]$ suffices. For all the cases were $n>0$ we use the following construction.

The functor $D$ is defined (see definition \ref{descent functor}) to have the following relation with a functor $K\of\simCat/[n]\to\ambient\s/\horn_k[n]$.
\[ \ambient\s/\simplex[n](\simplex(\xi), Df)\simeq \ambient\s/\horn_k[n](K(\xi), f) \]
This is possible because $\ambient\s/\simplex[n]$ is equivalent to the category of presheaves over $\simCat/[n]$ and the Yoneda lemma applies to those presheaves. 
The relation of $D$ and $K$ allows us to reduce the problem of lifting a horn against $Df$ to the problem of lifting some finite colimit of objects in the image of $K$ against $f$. These finite colimits have the required left lifting property by lemma \ref{left lifting property}.
\end{proof}

The rest of this section works out the definition and the properties of $K$.

\begin{definition}
The following defines the functor $K\of\simCat/[n]\to\ambient\s/\horn_k[n]$.
\begin{enumerate}
\item A function $\xi\of[m]\to[n]$ cuts $[m]$ into $n+1$ posets $\xi_j = \set{i\of[m]|\xi(i)=j}$. 
\item Let $\norm \xi$ be the number of elements of the product $\product{i\of ([n]-\set k)}{\xi_i}$. 
\item Define $K_0(\xi)\of [m+\norm\xi]\to [n]$ as follows.
\[ 
	K_0(\xi)(i) = \left\{
		\begin{array}{cc}
			\xi(i) & \xi(i)<k \\
			k & \xi(i-\norm\xi)\leq k \leq \xi(i)\\
			\xi(i-\norm\xi) & k<\xi(i-\norm\xi)
		\end{array}
	\right.
\]

\item In the \keyword{lexicographical product} $\product{i\of([n]-\set k)}{\xi_i}$ tuples get the \keyword{lexicographical ordering}. This ordering determines priority by comparing elements in sequence.
\begin{align*} 
(x_0,x_1,\dotsc) \leq_{\rm lex} (x_0,x_1,\dotsc) \iff& x_0\leq y_0 \land (x_0=y_0 \to (x_1,\dotsc)\leq_{\rm lex}(y_1,\dotsc))
\end{align*}

\item Let $\ka(\xi)\of\product{i\of([n]-\set k)}{\xi_i} \to [m+\norm\xi]$ be the \emph{nondecreasing} injection $\ka$ which sends the lexicographical product $\product{i\of([n]-\set k)}{\xi_i}$ to the interval in $[m+\norm\xi]$ which starts at the least $i$ such that $K_0(\xi)(i)=k$.

\item Let $\la(\xi)\of[m]\to[m+\norm\xi]$ be the unique nondecreasing injection which skips the image of $\ka$. This means $\la(i)=i$ if $\xi(i)<k$ and $\la(i)=i+\norm\xi$ if $\xi(i)\geq k$. Moreover $K_0(\xi)\circ\la(\xi) = \xi$.
\item For each morphism $\phi\of\xi\to\xi'$ in $\simCat/[n]$ let $\phi_i\of\xi_i\to\xi'_i$ be the fibrewise morphism and let $\product{i\of([n]-\set k)}{\phi_i}$ be the corresponding map of the (lexicographical) products $\product{i\of([n]-\set k)}{\xi_i}\to\product{i\of([n]-\set k)}{\xi'_i}$.
\item Let $K_1(\phi)\of K_0(\xi)\to K_0(\xi')$ be the unique nondecreasing function which satisfies $K_1(\phi)\circ \ka(\xi) = \ka(\xi')\circ \product{i\of([n]-\set k)}{\phi_i}$ and $K_1(\phi)\circ \la(\xi) = \la(\xi')\circ \phi$.
\end{enumerate}

The maps $(K_0,K_1)$ define an endofunctor of $\simCat/[n]$. To get the functor, let $K(\xi) = h\ri(\simplex(K_0(\xi)))$ for objects $\xi$ of $\simCat/[n]$ and $K(\phi) = h\ri(\simplex(K_1(\phi)))$ for morphisms, where $h\ri$ is the reindexing functor $\ambient\s/\simplex[n]\to\ambient\s/\horn_k[n]$ along the horn inclusion $h\of \horn_k[n]\to\simplex[n]$. Concretely, $K$ satisfies the following equations for all objects $\xi$ of $\simCat/[n]$, all $x\of \base(\dom(K(\xi)))$ and all arrows $\phi$ of $\simCat/[n]$.
\begin{align*}
\base(\dom(K(\xi))) &= \set{ x\of\base\simplex[m+\norm\xi]| K_0(\xi)\circ x\of\base\horn_k[n]}\\
K(\xi)(x) &= K_0(\xi)\circ x \\
\dim(x) &= \dom(x) \\
x\cdot\phi &= x\circ\phi \\
K(\phi)(x) &= K_1(\phi)\circ x
\end{align*}
\end{definition}

Thanks to the following property, the natural equivalence of homsets above extends to horn inclusions.

\begin{lemma} An \keyword{intersection} is a pullback square of monomorphisms. The functor $K$ preserves all intersections of $\simCat/[n]$. \end{lemma}

\begin{proof} The easiest way to see that is to describe the monomorphisms in the intersection as decidable predicates on the domain of a map $\xi\of[m]\to[n]$. Maps $[m]\to\bool$ that equals $\true$ at least once characterize monomorphisms to $\xi$. The action of the endofunctor $(K_0,K_1)$ on the monomorphism has a parallel on predicates. For each $p\of[m]\to\bool$ the following equations define $p^K\of [m+\norm\xi]\to\bool$. 
\begin{align*}
p^K(\kappa(\vec x)) &= \forall i\of[m]-\set k.p(x_i) &
p^K(\lambda(x)) &= p(x)
\end{align*}
It all works out because $\forall i\of[m]-\set k.p(x_i)\land q(x_i)$ is equivalent to $(\forall i\of[m]-\set k.p(x_i))\land(\forall i\of[m]-\set k.q(x_i))$.
\end{proof}

\begin{definition}
For each $f\of X\to\horn_k[n]$ the functor $D\of \ambient\s/\horn_k[n]\to\ambient\s/\simplex[n]$ satisfies the following equations.
\begin{align*}
\base(\dom(D(f))) &= \sum_{\xi\of\base\simplex[n]}(\ambient\s/\horn_k[n])(K(\xi),f)\\
\dim(\xi,x) &= \dom\xi\\
(\xi,x)\cdot\phi &= (\xi\circ\phi,x\circ K(\phi))\\
D(f)(\xi,x) &= \xi
\end{align*}
On morphisms in $\ambient\s/\horn_k[n]$, the functor $D$ acts by composition: $Dm(\xi,x) = (\xi,m\circ x)$. \label{descent functor}
\end{definition}

If $\ambient$ has infinite colimits, the functor $D$ has a left adjoint $K'$ that satisfies $K'\simplex\simeq K$. We cannot rely on that property here. The functor $K$ does extend to finite colimits of simplices. Preservation of intersections implies that $K$ also applies to finite unions of simplices like $\horn_l[m]$. The reason $D$ preserves fibrations and their right lifting property, is that its partial left adjoint preserves the left lifting property. 

\begin{definition}[Face notation] We introduce a notation for monomorphisms in $\simCat/[n]$ or more accurately for the result of applying $\simplex$ to them. For $\xi\of[m]\to[n]$ and $\phi\of[m]\to\bool$ let $\face(\xi,p)\of \dom(\face(\xi,p))\to \simplex[m]$ satisfy
\begin{align*}
\base(\dom(\face(\xi,p))) &= \set{\alpha\of\base(\simplex[m])\middle|\forall i\of \dom(\alpha).p\circ\alpha(i)=\true}\\
\dim(\alpha) &= \dom(\alpha)\\
\alpha\cdot\chi &= \alpha\cdot\chi\\
\face(\xi,p)(\alpha) &= \alpha
\end{align*}
\[\xymatrix{
\bullet\ar[rr]^{\face(\xi,p)}\ar[dr] && \simplex[m]\ar[dl]^{\simplex(\xi)} \\
& \simplex[n]
}\]
\end{definition}

With the help of this notation, we describe commutative triangles of monomorphisms as follows.
\begin{align}
&\horn_l[m]\stackrel{\hat h}\to \simplex[m]\stackrel{\simplex(\xi)}\to \simplex[n] = \bigvee_{i\of [m]-\set l} \face(\xi,\function x{x\neq i})\\
&K(\horn_l[m])\stackrel{K(\hat h)}\to \simplex[m+\norm\xi]\stackrel{\simplex(K_0(\xi))}\to \simplex[n] = \bigvee_{j\of [m]-\set l} \face(K_0(\xi),v_j)\\
&h\ri(K(\horn_l[m]))\stackrel{h\ri(K(\hat h))}\to h\ri(\simplex[m+\norm\xi])\stackrel{h\ri(\simplex(K_0(\xi)))}\to \horn_l[n] = \\
&\quad\left(\bigvee_{i\of[n]-\set k} \face(K_0(\xi),u_i)\right)\wedge\left(\bigvee_{j\of [m]-\set l} \face(K_0(\xi),v_j)\right)=\\
&\quad\left(\bigvee_{\substack{i\of[n]-\set k\\j\of [m]-\set l}}\face(K_0(\xi),u_i\land v_j)\right)\\
&h\ri(\simplex[m+\norm\xi])\stackrel{h}\to \simplex[m+\norm\xi]\stackrel{K_0(\xi)}\to \simplex[n] =\left(\bigvee_{i\of[n]-\set k} \face(K_0(\xi),u_i)\right)
\end{align}
\begin{center}
where
\end{center}
\begin{align*}
u_i(x) &= (K_0\xi(x)\neq i) & 
v_i(\kappa(\xi)(\vec x)) &= (\xi(i)= k\vee x_{\xi(i)}\neq i)\\
v_i(\lambda(\xi)(x)) &= (x\neq i)&
u_i\land v_i(x) &= u_i(x)\land v_i(x)
\end{align*}

\begin{lemma} The monomorphism $h\ri K(\hat h)\of$
\[ \left(\bigvee_{\substack{i\of[n]-\set{k}\\j\of [m]-\set l}}\face(K_0(\xi),u_i\land v_j)\right)\to\left(\bigvee_{i\of[n]-\set{k}} \face(K_0(\xi),u_i)\right) \] 
has the left lifting property with respect to all fibrations.\label{left lifting property}\end{lemma}

\begin{proof}The class of morphisms with the left lifting property is closed under compositions and pushouts. The monomorphism $h\ri K(\hat h)$ belongs to this class because of this closure property.

Decompose $h\ri K(\hat h)$ as the inclusions of the following three subobjects.
\begin{align*}
A &=\dom(h\ri K(\hat h))=\left(\bigvee_{\substack{i\of[n]-\set{k}\\j\of [m]-\set l}}\face(K_0(\xi),u_i\land v_j)\right)\\
B &=\left(\bigvee_{j\of[m]-\set l}\face(K_0(\xi),u_{\xi(l)}\land v_j)\right)\vee \left(\bigvee_{i\of[n]-\set{k,\xi(l)}}\face(K_0(\xi),u_i)\right)\\
C &=\cod(h\ri K(\hat h))=\left(\bigvee_{i\of[n]-\set{k}} \face(K_0(\xi),u_i)\right)
\end{align*}

Decompose $A\to B$ as a series of inclusions $A_j\to A_j+1$ that satisfy $A_0=A$ and $A_{n+1}=B$. For $j\of[n+1]$ define $A_j$ as follows.
\[ A_j = A\vee \left(\bigvee_{i\of[j]-\set{k,j,\xi(l)}}\face(K_0(\xi),u_i)\right) \]
Each inclusion $A_j\to A_{j+1}$ is a pushout of the inclusion of $A_j\cap \face(K_0(\xi),u_{j+1})$ into $\face(K_0(\xi),u_{j+1})$, where
\begin{align*} A_j\cap \face(K_0(\xi),u_{j+1}) = &\left(\bigvee_{i\of[m]-\set l}\face(K_0(\xi),u_{j+1}\land v_{i})\right)\vee \\
&\left(\bigvee_{i\of[j+1]-\set{k,\xi(l)}}\face(K_0(\xi),u_i\land u_{j+1})\right)\end{align*}
The faces $\face(K_0(\xi),u_{j+1}\land v_{i})$ and $\face(K_0(\xi),u_i\land u_{j+1})$ all contain the point $\lambda(l)$ of $\simplex[m+\norm\xi]$ which means that lemma \ref{face completion} applies to the inclusion $A_j\land \face(K_0(\xi),u_{j+1}) \to \face(K_0(\xi),u_{j+1})$. Hence $A\to B$ has the left lifting property.

If $\xi(l)=k$, then $B=C$ and the proof is complete. Suppose $\xi(l)=i\neq k$. First define the following maps $[m+\norm\xi]\to\bool$ for each $\vec p \of\product{i\of [n]-\set k}{\xi_i}$.
\begin{align*}
w_{\vec p}(\kappa(\vec q)) &= \true &
w_{\vec p}(\lambda(q)) &= ((\xi(q)=k\vee q\leq p_{\xi(q)})\land\xi(q)\neq \xi(l))
\end{align*}

Decompose $B\to C$ as a series of inclusions $B_j\to B_j+1$ that satisfy $B_0=B$ and $B_{m+\norm\xi+1}=C$. For $j\of[m+\norm\xi+1]$ define $B_j$ as follows.
\[ B_j = B \vee \left(\bigvee_{\kappa(\vec p)<j} \face(K_0(\xi),u_{\xi(l)}\land w_{\vec p})\right)\]

When $j\neq\kappa(\vec p)$ for any $\vec p \of \product{i\of [n]-\set k}{\xi_i}$, the inclusion $B_j\to B_{j+1}$ is the identity arrow, which satsifies the left lifting property trivially. 

Suppose $j\neq\kappa(\vec p)$. Each inclusion $B_j\to B_{j+1}$ is a pushout of the inclusion of $B_j\cap \face(K_0(\xi),u_{\xi(l)}\land w_{\vec p})$ into $\face(K_0(\xi),u_{\xi(l)}\land w_{\vec p})$, where
\begin{align*} B_j\cap \face(K_0(\xi),w_{\vec p}) = 
&\bigvee_{i\of [n]-\set{\xi(l),k}}\face(K_0\xi,u_i\land w_{\vec p})\vee\\
&\bigvee_{i\of [m]-\set l} \face(K_0\xi,v_i\land w_{\vec p})\vee\\
&\bigvee_{\kappa(\vec q)<j} \face(K_0\xi,w_{\vec q}\land w_j)
\end{align*}
Let $\vec p[l]\of \product{i\of [n]-\set k}{\xi_i}$ satisfy $p[l]_j = l$ if $\xi(l)=j$ and $p[l]_i=p_i$ otherwise. The point $\kappa(p[l])$ is a member of $\face(K_0\xi,u_i\land w_{\vec p})$ for all $i\of [n]-\set{\xi(l),k}$. It is also a member of $\face(K_0\xi,w_{\vec q}\land w_{\vec p})$ for all $\vec q\of \product{i\of [n]-\set k}{\xi_i}$ such that $\kappa(\vec q)<\kappa(\vec p)$. When $\kappa(\vec p[l])$ is not a member of $\face(K_0\xi,v_i\land w_{\vec p})$ for some $i$, then $\face(K_0\xi,v_i\land w_{\vec p})$ is a subobject of a face of $B_j\cap \face(K_0(\xi),\land w_{\vec p})$ that does contain $\kappa(\vec p[l])$.

Let $i\of[m]-\set l$. If not $v_i(\kappa(\vec p[l]))$ then by definition $\xi(i)\neq k$ and $(p[l]_\xi(i) = i)$. If $i$ is the least member of $\xi_{\xi(i)}$, then $\face(K_0\xi,v_i\land w_{\vec p})\subseteq \face(K_0\xi,u_{\xi(i)}\land w_{\vec p})$. If $\xi(i-1)=\xi(i)$, then $\kappa(\vec p[i-1])<\kappa(\vec p)$ and $\face(K_0\xi,v_i\land w_{\vec p})\subseteq \face(K_0\xi,W_{\vec p[i-1]}\land w_{\vec p})$.

Since $B_j\cap \face(K_0(\xi),\land w_{\vec p})$ is a unions of faces that have the point $\vec p[l]$ in common, lemma \ref{face completion} applies to the inclusion $B_j\cap \face(K_0(\xi),\land w_{\vec p})\to \face(K_0(\xi),\land w_{\vec p})$. Therefore $B_j\to B_{j+1}$, $B\to C$ and $h\ri K(\hat h)\of A\to C$ all have the left lifting property.
\end{proof}

\begin{lemma}[Face completion] Let $F$ be an inhabited decidable set of faces of $\simplex[p]$ which all have a point $e$ of $\simplex[p]$ in common. The inclusion $\bigcup F\to \simplex[p]$ has the left lifting property with respect to all fibrations. \label{face completion} \end{lemma}

\begin{proof} For all $j\of[p]$ let $F_j$ be the union of $F$ with the set of $j$-dimensional faces of $\simplex[p]$ which contain the edge $e$. Because $F$ is inhabited, $\bigcup F$ contains $e$ and therefore $F_0=F$. Because $\simplex[p]$ is a $p$-dimensional face of $\simplex[p]$ which contains $e$, $\bigcup F_p = \simplex[p]$. For $j>0$ let $S_j$ be the set of $j$-dimensional faces of $\bigcup F_j$ which are not already contained in $\bigcup F_{j-1}$. If a $j$-dimensional face $\face(\Sigma)$ of $\bigcup F_j$ opposes $e$, it is part of a higher dimensional face which is a member of $F$. Therefore each face $\face(\Sigma)\of S_j$ contains $e$. For this reason $\face(\Sigma)\cap \bigcup F_{j-1}$ is the horn whose central edge is $e$. The inclusion $\bigcup F_{j-1}\to\bigcup F_j$ is therefore the pushout of a coproduct of horn inclusions indexed over $S_j$. Therefore the inclusion has the left lifting property. Because composition preserves the left lifting property, $\bigcup F = F_0\to F_p = \simplex[p]$ has it too. 
\end{proof}

\printbibliography

\end{document}